\documentclass{amsart}

\usepackage[T1]{fontenc}
\usepackage[utf8x]{inputenc}
\usepackage{enumerate, amsmath, amsfonts, amssymb, amsthm, mathrsfs, wasysym, graphics, graphicx, xcolor, url, hyperref, hypcap, xargs, multicol, pdflscape, multirow, hvfloat, array, ae, aecompl, pifont, mathtools, a4wide, float, blkarray, overpic, nicefrac, etoolbox, stmaryrd, shuffle}
\usepackage[bb=boondox]{mathalfa}%allows for mathbb 0 and 1
\usepackage[shortlabels, inline]{enumitem}
\usepackage[noabbrev,capitalise]{cleveref}
\usepackage[normalem]{ulem}
\usepackage{marginnote}
\hypersetup{colorlinks=true, citecolor=darkblue, linkcolor=darkblue}
\usepackage[all]{xy}
\usepackage{tikz}
\usepackage{tikz-cd}
\usetikzlibrary{trees, decorations, decorations.pathmorphing, decorations.markings, decorations.shapes, shapes, arrows, matrix, calc, fit, intersections, patterns, angles}
\graphicspath{{figures/}{figures/nodes/}}
\makeatletter\def\input@path{{figures/}}\makeatother
\usepackage{caption}
\usepackage[export]{adjustbox}

\newtheorem*{theorem*}{Theorem}%[section]

\theoremstyle{definition}

\crefname{equation}{Equation}{Equations}

\newcommand{\set}[2]{\left\{ #1 \;\middle|\; #2 \right\}} % set notation
\newcommand{\ie}{\textit{i.e.}~} % id est
\newcommand{\eg}{\textit{e.g.}~} % exempli gratia
\definecolor{darkblue}{rgb}{0,0,0.7} % darkblue color
\definecolor{green}{RGB}{57,181,74} % green color
\definecolor{violet}{RGB}{147,39,143} % violet color
\newcommand{\darkblue}{\color{darkblue}} % darkblue command
\newcommand{\defn}[1]{\textsl{\darkblue #1}} % emphasis of a definition
\newcommand{\OEIS}[1]{\cite[{\rm \href{http://oeis.org/#1}{\texttt{#1}}}]{OEIS}}

\title[Plumbing bijections]{\includegraphics{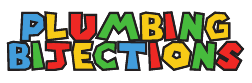}}

\thanks{
Partially supported by the Spanish project PID2022-137283NB-C21 of MCIN/AEI/10.13039/501100011033 / FEDER, UE, by the Spanish--German project COMPOTE (AEI PCI2024-155081-2 \& DFG 541393733), by the Severo Ochoa and María de Maeztu Program for Centers and Units of Excellence in R\&D (CEX2020-001084-M), by the Departament de Recerca i Universitats de la Generalitat de Catalunya (2021 SGR 00697), and by the French--Austrian project PAGCAP (ANR-21-CE48-0020 \& FWF I 5788).
}

\author{Vincent Pilaud}
\address[Vincent Pilaud]{Universitat de Barcelona \& Centre de Recerca Matemàtica, Barcelona, Spain}
\email{vincent.pilaud@ub.edu}
\urladdr{\url{https://www.ub.edu/comb/vincentpilaud/}}

\begin{document}

\begin{abstract}
The legendary Mario and Luigi show us that whether you slap in the crossings as early as a warp pipe can shoot you or as late as the very last bend, the water system in Yoshi Hill comes out exactly the same!
\end{abstract}

\vspace*{-2cm}
\maketitle

%%%%%%%%%%%%%%%%%%%%%%%%%%%%%%%%%%%%%%%

%\section{Water Emergency at Yoshi Hill!}
\section*{\centerline{\includegraphics[scale=.8]{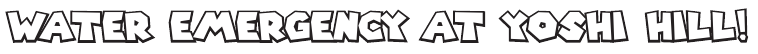}}}

\textsl{``Mamma mia!''} cry Mario and Luigi as they zoom to the scene. The Monty Moles have totally wrecked the pipes, leaving geysers and puddles everywhere! (see~\cref{fig:YoshiHill1} left).

\begin{figure}[h]
	\centerline{\includegraphics[scale=.14]{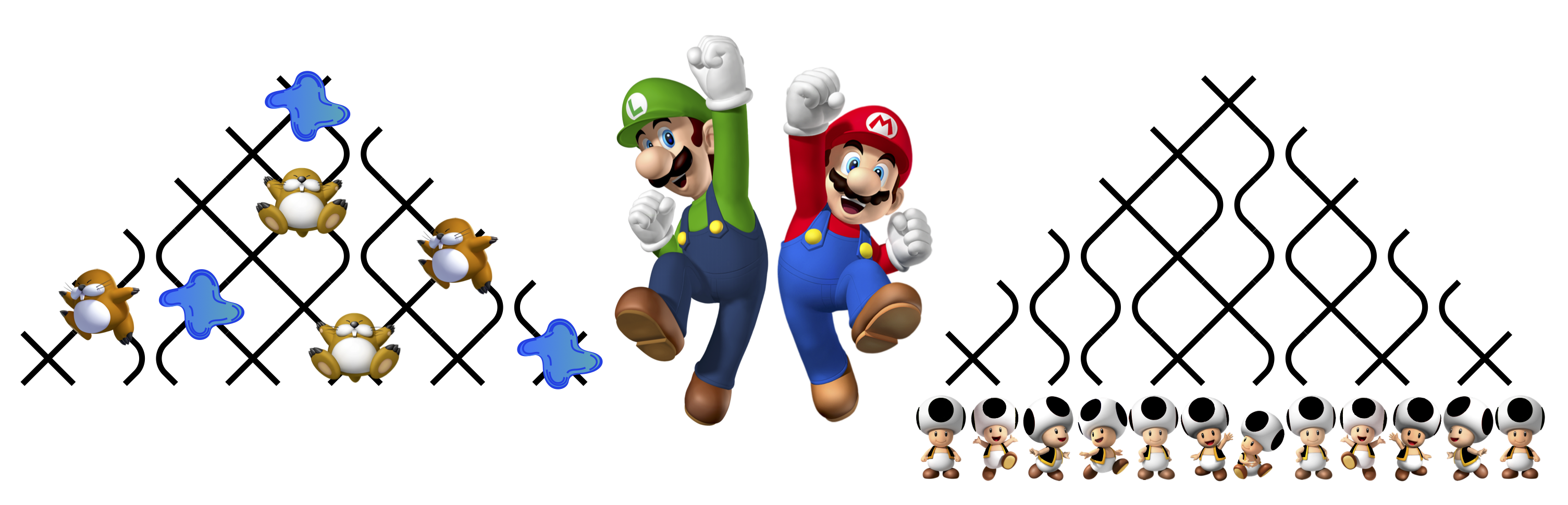}}
	\caption{Water emergency at Yoshi Hill (left). Mario and Luigi's fix (right). But which Toad is plugged into which water supply?}
	\label{fig:YoshiHill1}
\end{figure}

Our famous plumbers jump into action, snapping elbowed or crossing pipes into every mole-dug hole.
Pipes twist, water splashes, and the residents cheer as the chaos subsides… until they see their water bills!
Donkey Kong ends up paying for Yoshi, Yoshi for Princess Peach, and so on.
The Mushroom Water Company dispatches a Toad repair squad to every home, but the whole situation remains a twisted tangle of pipes and headaches (see~\cref{fig:YoshiHill1} right).
Mario recalls: \textsl{``I put that elbow at the intersection of the first uphill street with the third downhill avenue!''}
Luigi objects: \textsl{``Nope! At the intersection of the third uphill street with the first downhill avenue! Guaranteed!''}
Reality check: their memories are as twisted as the pipes.

The big boss of the Mushroom Water Company blows his top and yells like a Bob-omb~about~to explode:
\textsl{``Who should get these bills?''}
Mario and Luigi exchange glances and answer:
\textsl{``Easy!~The $i$th pipe from the left (resp.~right) clearly reaches one of the first (resp.~last) $2i$ customers.''}
Toadsworth pops up, sweating under his mushroom cap: \textsl{``Whoa! But there are 198\,272 possibilities!''}

Deep breath. Toadsworth bites his lip, picks one option… and prays the water flows to the right homes.
Honestly, it would’ve been way simpler to just flush some colored dye down each pipe! (see~\cref{fig:YoshiHill2} left).

\begin{figure}
	\centerline{\includegraphics[scale=.14]{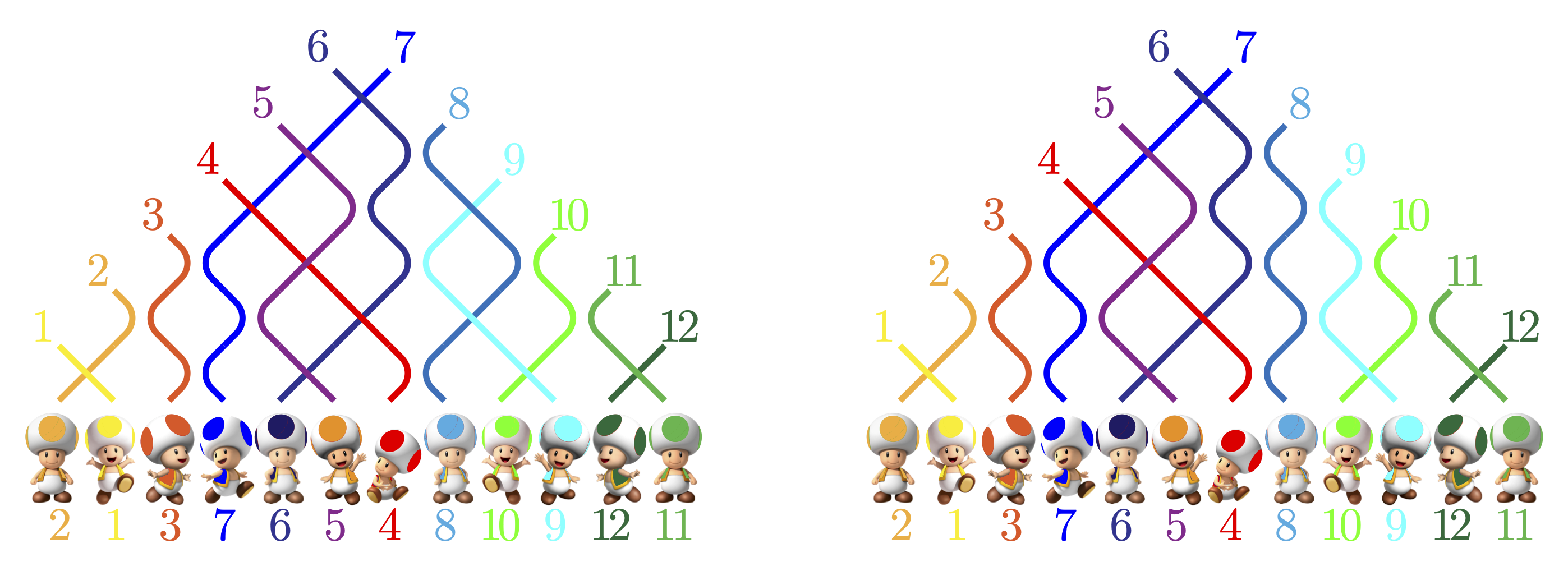}}
	\caption{The new pipe flow in Yoshi Hill (left), and a tidier solution with fewer crossings (right).}
	\label{fig:YoshiHill2}
\end{figure}

Too bad! Mario and Luigi zoom back to Yoshi Hill to rebuild the pipes so they match the configuration Toadsworth picked.
Very quickly, they both realize it’s doable starting from the faucets and installing elbows and crossings until the water reaches all the customers. Since crossings are trickier to install than elbows, they agree to use as few crossings as possible (see~\cref{fig:YoshiHill2} right).
But even with this rule, a whole bunch of options pop up like Piranha Flowers!
And, as always, our heroic duo can’t help picking totally different strategies:
Motivated Mario wants to plant the crossings right at the start, while Lazy Luigi prefers to leave them for the very last twist of the pipes!
(see~\cref{fig:YoshiHill3}).

Being the best of brothers, they heatedly debate their approach:
\textsl{``In my plan,'' }says Mario, puffing up his chest, \textsl{``I would never put a crossing between two pipes that have already brushed elbows. And if a diamond fits in Yoshi Hill, I’ll never place elbows at its left and right vertices without putting one at the bottom vertex!''}
\textsl{``In mine,''} Luigi retorts, scratching his head under his cap, \textsl{``I would never put a crossing between two pipes that will brush elbows later. And I will never place elbows at the left and right vertices of your diamond without adding one at the top vertex!''}

\begin{figure}[b]
	\centerline{\includegraphics[scale=.14]{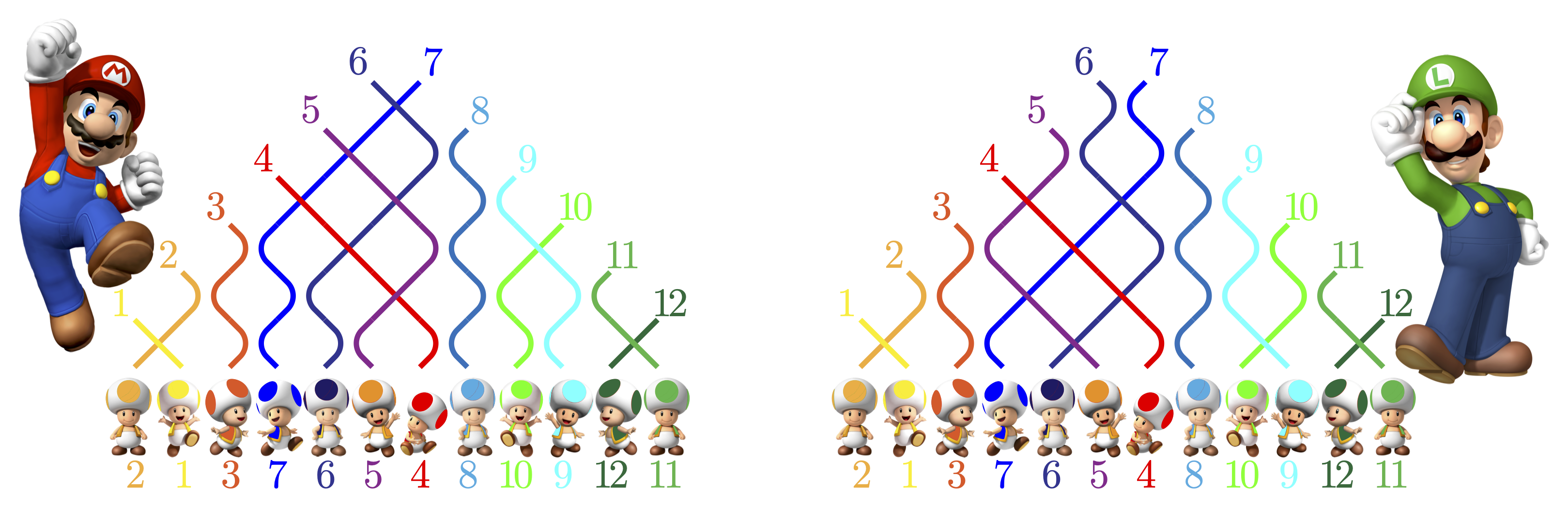}}
	\caption{Mario's (left) and Luigi's (right) solutions for rebuilding the pipe system of Yoshi Hill.}
	\label{fig:YoshiHill3}
\end{figure}

Toadsworth puts an end to the argument.
\textsl{``The important thing,''} he says, wiggling his mushroom cap, \textsl{``is that everyone finds the pipe that matches their bill.''}
As he’s really packing power in his mushroom, he adds thoughtfully:
\textsl{``Actually, what Mario and Luigi have taught us is that there are plumbing bijections between bottom diamond closed pipe configurations, top diamond closed pipe configurations, and Yoshi bill assignments.''}

Professor E. Gadd, always eager to turn simple things into complicated lectures, adjusts his glasses and declares in his grand, scholarly voice:
\textsl{``Toadsworth, if we encode a pipe configuration as the pairs (a,b) where an elbow sits at the intersection of the $a$th uphill street with the $b$th downhill avenue, and we encode a bill assignment as the permutation~$\sigma$ of~$[n]$ such that the $i$th bill goes to the $\sigma(i)$th customer, our plumbers suggest (size preserving) bijections between...
\begin{enumerate}[(a)]
\item \defn{intersection closed collections} of subintervals of~$[n]$ (\ie subsets~$X$ of~$\set{(a,b)}{1 \le a \le b \le n}$ such that $[a,c] \in X$ and~$[b,d] \in X$ implies~$[b,c] \in X$ for all~$1 \le a \le b \le c \le d$),
\item \defn{union closed collections} of subintervals of~$[n]$ (\ie subsets~$X$ of~$\set{(a,b)}{1 \le a \le b \le n}$ such that $[a,c] \in X$ and~$[b,d] \in X$ implies~$[a,d] \in X$ for all~$1 \le a \le b \le c \le d$),
\item \defn{Yoshi permutations} of~$[2n]$ (\ie with $\sigma(i) \le 2i$ and~$\sigma(2n-i+1) \ge 2(n-i)+1$ for~$i \in [n]$),
\end{enumerate}
which are famously counted by median Genocchi numbers!''} \OEIS{A005439} (Up to appropriate compositions, the Yoshi permutations are the permutations~$\tau$ of~$[2n]$ with~$\tau(2i-1) \ge 2i-1$ and~$\tau(2i) \le 2i$ for all~$i \in [n]$, which were considered in~\cite{DumontRandrianarivony}.)

And, true to his habit of turning a molehill into a mountain, he keeps thinking out loud: \textsl{``Of course, the same plumbing techniques obviously give (size preserving) bijections between...
\begin{enumerate}[(a)]
\item \defn{crossing internally closed} graphs on~$[n]$ (\ie such that~$(a,c) \in G$ and~$(b,d) \in G$ implies~$(b,c) \in G$ for all~$1 \le a < b < c < d \le n$),
\item \defn{crossing externally closed} graphs on~$[n]$ (\ie such that~$(a,c) \in G$ and~$(b,d) \in G$ implies~$(a,d) \in G$ for all~$1 \le a < b < c < d \le n$),
\item \defn{Yoshi derangements} of~$[2n]$ (\ie with $\sigma(i) < 2i$ and~$\sigma(2n-i+1) > 2(n-i)+1$ for~$i \in [n]$),
\end{enumerate}
which are again counted by median Genocchi numbers!''} % \OEIS{A005439}.
But he is suddenly interrupted in the middle of his wild musings...

%%%%%%%%%%%%%%%%%%%%%%%%%%%%%%%%%%%%%%%

%\section{Water Emergency at Plumbopolis!}
\section*{\centerline{\includegraphics[scale=.8]{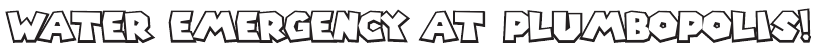}}}

\enlargethispage{.2cm}
The exploits of Mario and Luigi have certainly made them famous!
They zoom off to Plumbopolis, and once again, chaos reigns.

They ask the townsfolk for a map, and the residents reply:
\textsl{``I'm always driving with my head on the wheel, I have no clue what the city actually looks like!''}
After gathering a few scattered descriptions, Mario and Luigi finally piece it together: \textsl{``Hmm… looks like Plumbopolis is a staircase polyomino, with a leftmost and a rightmost corner linked by two step-like paths.''} (See \cref{fig:plumbopolis}.)

\begin{figure}[h]
	\centerline{\includegraphics[scale=1.8]{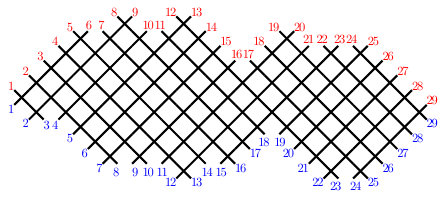}}
	\caption{Plumbopolis is just like any staircase polyomino city. It defines a permutation~$\tau$ of~$[29]$ obtained by following the uphill streets or downhill avenues, \eg~$\tau({\color{red} 5}) = {\color{blue} 14}$ and~$\tau({\color{red} 15}) = {\color{blue} 7}$.}
	\label{fig:plumbopolis}
\end{figure}

Again, Toadsworth grabs a bill assignment at random and cries out for Mario and Luigi to straighten the pipes!
Just like before, the brothers go head-to-head, each defending their own strategy to fix the water system.
Long-story short, you certainly know what Mario and Luigi have discovered:

\begin{theorem*}
Consider any staircase polyomino city with $m$ faucets and $m$ custumers, and let~$\tau$ be the permutation of~$[m]$ such that the grid water system (meaning only crossings) sends the water of the $i$th faucet to the $\tau(i)$th custumer.
Then plumbing provides (size preserving) bijections between
\begin{enumerate}[(a)]
\item subsets~$S$ of road intersections such that if~$S$ contains the left and right vertices of a diamond that fits in the city, then~$S$ contains also its bottom vertex,
\item subsets~$S$ of road intersections such that if~$S$ contains the left and right vertices of a diamond that fits in the city, then~$S$ contains also its top vertex,
\item permutations of~$[m]$ lower than~$\tau$ in (strong) Bruhat order (recall that~$\sigma \le \tau$ if~$\sigma$ admits a reduced expression which is a subword of a reduced expression for~$\tau$).
\end{enumerate}
\end{theorem*}

\begin{proof}
The water system defines a word~$Q$ on simple transpositions, whose Demazure product is~$\tau$.
For any permutation~$\sigma$ below~$\tau$ in Bruhat order, the corresponding subword complex~$SC(Q,\sigma)$ (see~\cite{KnutsonMiller-GroebnerGeometry,KnutsonMiller-subwordComplex}) is thus non-empty and admits a unique greedy facet and a unique antigreedy facet (see~\cite{Pilaud-greedyFlipTree,PilaudStump-ELlabeling}).
It is elementary to check that the greedy facet satisfies~(i) while the antigreedy facet satisfies (ii).
\end{proof}

\begin{figure}[b]
	\centerline{\includegraphics[scale=.09]{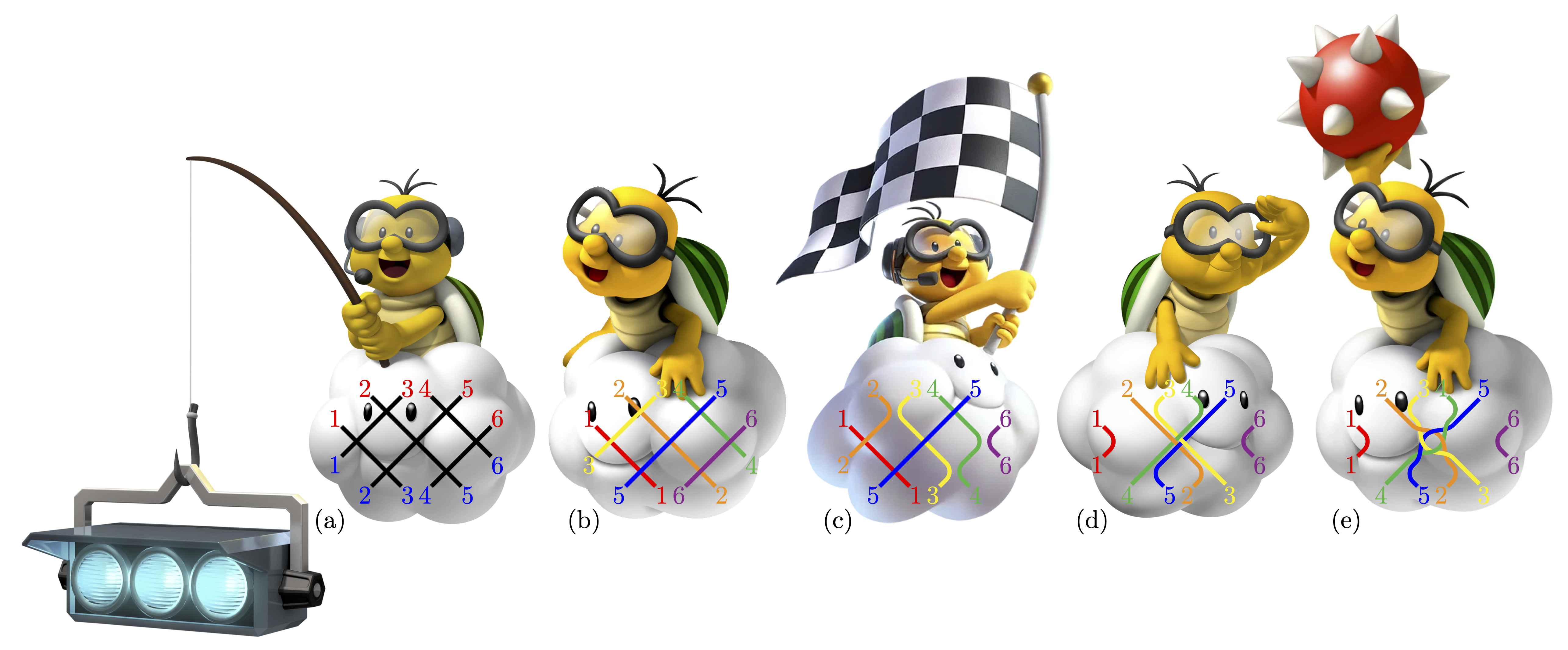}}
	\caption{Lakitu Island. Here, we have~$\tau = 351624$, $\mu = 111224$ and $\nu = 355666$. The two permutations~$\sigma = 351624$ and~$\sigma' = 145236$ both satisfies~$\mu(i) \le \sigma(i) \le \nu(i)$ for all~$i \in [6]$, but only~$\sigma$ is below~$\tau$ in (strong) Bruhat order.}
	\label{fig:lakituIsland}
\end{figure}

\enlargethispage{.5cm}
It is clear that any permutation~$\sigma$ of~(c) satisfies that~$\mu(i) \le \sigma(i) \le \nu(i)$ for all~$i \in [m]$, where~$\mu(i)$ and~$\nu(i)$ are the leftmost and rightmost customers reachable from the $i$th faucet (for example~${\mu({\color{red} 19}) = {\color{blue} 11}}$ and~$\nu({\color{red} 19}) = {\color{blue} 27}$ in \cref{fig:plumbopolis}).

The reciprocal statement does not hold in general.
Luckily, Toadsworth has hidden talents, and he already knew about this subtlety thanks to last year’s disasters on Lakitu Island, where $\tau = 351624$, $\mu = 111224$ and $\nu = 355666$ (see \cref{fig:lakituIsland}~(a~\&~b)).
Back then, he had asked Mario and Luigi to handle several bill assignments. Our two plumbers managed to pull it off for the permutation~$\sigma = 351624$ (see \cref{fig:lakituIsland} (c)), but they eventually ended up with a full plate of pipe-spaghetti for the permutation~$\sigma' = 145236$ (see \cref{fig:lakituIsland}~(d~\&~e)).
Both permutations~$\sigma$ and~$\sigma'$ satisfies~$\mu(i) \le \sigma(i) \le \nu(i)$ for all~$i \in [6]$, but only~$\sigma$ is below~$\tau$ in (strong) Bruhat order.

Using that~${\sigma \le \tau}$ in Bruhat order if and only if ${\#\set{k \le i}{\sigma(k) \le j} \ge \#\set{k \le i}{\tau(k) \le j}}$ for all~${i,j \in [n]}$, it is easy to check that the reciprocal statement does hold as soon as the top path has a unique peak (as in Yoshi hill) or the bottom path has a unique valley.

%%%%%%%%%%%%%%%%%%%%%%%%%%%%%%%%%%%%%%%

\addtocontents{toc}{\vspace{.1cm}}
\section*{Notes and acknowledgments}

The crossing internally closed graphs (a) are also known as \defn{grounded rectangle graphs}~\cite{JelinekTopfer}, \defn{hook graphs}~\cite{Hixon}, \defn{max point-tolerance graphs}~\cite{CatanzaroChaplickFelsnerHalldorssonHalldorssonHixonStacho}, \defn{$p$-box graphs}~\cite{SotoCaro}, or \defn{non-jumping graphs}~\cite{AshurFiltserSababn}, while the crossing externally closed graphs (b) are also known as \defn{terrain-like graphs}~\cite{FroeseRenken-algo,FroeseRenken,AshurFiltserSababn}.
A similar bijection from (b) to Dumont derangements (c) was presented in~\cite{FroeseRenken} (but definitely more twisted than a Mario pipe maze!), but I am not aware that the bijection from (a) to either (b) or (c) was observed earlier, even if the two classes of graphs (a) and (b) were compared in~\cite{AshurFiltserSababn}.

The sequence \OEIS{A005439} appears twice in the context of interval hypergraphic polytopes~\cite{BergeronPilaud}: it counts the interval hypergraphs whose hypergraphic polytope is a nestohedron (closed under union), and the interval hypergraphs whose hypergraphic poset is a lattice (closed under intersection).
I am truly grateful to Andrew Sack and to Félix Gélinas respectively for bringing up this ``coincidence'' to my attention.
I would like to extend my sincere thanks to all members of the project ``Hypergraphic polytopes'' (Jose Bastidas, Félix Gélinas, Germain Poullot, Andrew Sack, Eleni Tzanaki) of the Workshop on Lattice Theory held in the Banff International Research Station in January 2025, and to all members of the project ``Exploring hypergraphic polytopes'' (Egor Bakaev, Julie Curtis, Félix Gélinas, Leonardo Saud Maia Leite, Juan Luis Valerdi, Yirong Yang) of the Intensive Research Program Combinatorial Geometries and Geometric Combinatorics held at the Centre de Recerca Matemàtica of Barcelona in October-November 2025.

I am also grateful to Federico Ardila, Jean Cardinal, and Christian Stump for discussions which happened at the CIRM research school ``Beyond Permutahedra and associahedra'' in December 2025.

I thank Germain Poullot for generously sharing his Mario expertise, helping me navigate the twists and turns like a pro plumber.
I am indebted to ChatGPT for providing Mario-style english phrasing, turning pipes, jumps, and warp zones into words on the page.

%%%%%%%%%%%%%%%%%%%%%%%%%%%%%%%%%%%%%%%

\bibliographystyle{alpha}
\bibliography{MarioLuigi}
\label{sec:biblio}

\end{document}